\documentclass[11pt]{article}
\usepackage{amsfonts, amsmath, amssymb, amsgen, amsthm, amscd,
latexsym}
\usepackage{color}
\usepackage[all]{xy}

\def\Hom{\mathop{\rm Hom}\nolimits}

\def\Cb{{\mathbb C}}

\def\Hc{{\mathcal{H}}}

\def\Fc{{\cal F}}

\def\Uc{{\cal U}}

\def\a{\alpha}

\def\d{\delta}
\def\D{\Delta}

\def\s{\sigma}

\def\ve{\varepsilon}
\def\vp{\varphi}

\def\ot{\otimes}

\def\acl{\blacktriangleright\hspace{-4pt}\vartriangleleft }

\def\0D{\Delta^{(0)}}
\def\1D{\Delta^{(1)}}

\newtheorem{theorem}{Theorem}[section]
\newtheorem{remark}[theorem]{Remark}
\newtheorem{proposition}[theorem]{Proposition}
\newtheorem{lemma}[theorem]{Lemma}

\newtheorem{example}[theorem]{Example}
\newtheorem{definition}[theorem]{Definition}

\def\build#1_#2^#3{\mathrel{
\mathop{\kern 0pt#1}\limits_{#2}^{#3}}}
\newcommand{\ps}[1]{~\hspace{-4pt}^{^{(#1)}}}
\newcommand{\ns}[1]{~\hspace{-4pt}_{_{{<#1>}}}}

\newcommand{\lrbicross}{{\blacktriangleright\!\!\!\triangleleft}}

\def\odots{\ot\cdots\ot}

\numberwithin{equation}{section}

 \newcommand{\ie}{{\it i.e.\/}\ }

\def\a{\alpha}

\def\d{\delta}

\def\s{\sigma}

\def\ve{\varepsilon}

\def\vp{\varphi}

\def\D{\Delta}

\def\ot{\otimes}
\def\part{\partial}

\def\text{\hbox}

\def\ot{\otimes}

\def\Hom{\mathop{\rm Hom}\nolimits}

\def\build#1_#2^#3{\mathrel{
\mathop{\kern 0pt#1}\limits_{#2}^{#3}}}

\numberwithin{equation}{section}
\newcommand{\comment}[1]{\relax}

\begin{document}
\title{\bf New coefficients for Hopf cyclic cohomology }

\author{Mohammad Hassanzadeh \\ University of Windsor\\ mhassan@uwindsor.ca     }

\date{}
\maketitle

\begin{abstract}
  In this note the categories  of coefficients for Hopf cyclic cohomology of comodule algebras and comodule coalgebras are extended. We show that  these new categories have two proper different subcategories where the smallest one is the known  category of stable anti Yetter-Drinfeld modules.  We prove that components of Hopf cyclic cohomology such as cup products work well  with these new coefficients.
\end{abstract}


\section*{Introduction}
 Coefficients for Hopf cyclic cohomology was first introduced in  \cite{HaKhRaSo1} under the name of stable anti Yetter-Drinfeld (SAYD) modules for all four quantum symmetries i.e., (co)module (co)algebras.
 These coefficients are  generalized  in different ways; first for bialgebra cyclic homology in \cite{atabey}, for Hopf algebras in \cite{staic},  in terms of contramodules in \cite{br}, for Hopf algebroids in \cite{bs2}, with categorical approach in \cite{bs1}    and finally   for module algebras and module coalgebras   in \cite{hkr}.  In  the last reference a  suitable class of coefficients for Hopf cyclic cohomology is  introduced  and  is  shown  to be much  larger than the category of SAYD modules.
More precisely,  it was  shown that there are at least three noticeable categories of such coefficients by which the Hopf cyclic cohomology  of module algebras and module coalgebras make sense. This new class of coefficients  are indispensable    as there are  Hopf algebras, such as Connes-Moscovici Hopf algebra which lack  a large  class of SAYD modules.
In fact the new  coefficients,   on the contrary to the old ones, depend on  both  Hopf algebra  and   (co)algebra  in question.   By introducing several examples it is also shown that these extensions of coefficients   are proper.

   \medskip

In this paper we   complete  \cite{hkr} by introducing  the new categories of coefficients for Hopf cyclic cohomology with respect to other two symmetries of comodule algebras and comodule coalgebras which were remained open in \cite{hkr}. One notes that Hopf cyclic cohomology of module coalgebras  generalizes the  Connes-Moscovici  Hopf cyclic cohomology  defined in \cite{ConMos:HopfCyc}. The Hopf cyclic cohomology  of module algebras generalizes the cyclic cohomology of algebras and twisted cyclic cohomology. In the case of comodule algebras one obtains the suitable coefficients for dual of Connes-Moscovici Hopf cyclic cohomology defined in \cite{bm1}. At the end  we show that the components of Hopf cyclic cohomology such as  cup products  \cite{atabey2}, \cite{bm2} work well with these new coefficients.\\

\medskip

\section*{Acknowledgement } The author would like to thank mathematics and statistics department of university of New Brunswick where major part of the work was done. Also the author appreciates  mathematics and statistics department of university of Windsor and  the Institut des Hautes
\'Etudes Scientifiques, IHES, where some portions of the work have been completed.\\


\textbf{Notations}: All Hopf algebras in this paper have bijective antipodes. We denote a Hopf algebra by $\Hc$ and  the counit of a Hopf algebra by $\ve$. We use the Sweedler summation notation $\Delta(h)= h\ps{1}\ot h\ps{2}$ for the coproduct of a Hopf algebra. Furthermore  $\blacktriangledown(h)= h\ns{-1}\ot h\ns{0}$ and $\blacktriangledown(h)= h\ns{0}\ot h\ns{1}$  are used for the left and right coactions of a coalgebra, respectively.


\tableofcontents

\section{Generalized Hopf cyclic cohomology with coefficients}
In this section we introduce two categories of suitable coefficients for Hopf cyclic cohomology for  comodule algebras and comodule coalgebras.   Let us recall from \cite{HaKhRaSo1} that a right-left SAYD module $M$ over a Hopf algebra $\mathcal{H}$ is a right module and a left comodule over $\mathcal{H}$ satisfying the following conditions.
\begin{align}
&\blacktriangledown(mh)= S(h^{(3)}) m\ns{-1}h^{(1)}\ot m\ns{0}h^{(2)}, \quad \text{(AYD condition)}\\
&m\ns{0}m\ns{-1}=m, \quad \text{(Stability condition)}.
\end{align}

\subsection{The $^A\mathcal{H}$-SAYD and HCC modules for comodule algebras}

In this subsection we introduce a generalization of  Hopf cyclic cohomology with coefficients respecting to the symmetry of   a comodule algebra.
Let  $M$ be a right-left SAYD module over $\mathcal{H}$ and $A$ be a left $\mathcal{H}$-comodule algebra. We let $\mathcal{H}$ coacts on $A^{\ot (n+1)}$ diagonally from left, \ie
\begin{equation*}
a_0\ot \cdots \ot a_n\longmapsto a_0\ns{-1}\cdots a_{n}\ns{-1} \ot a_0\ns{0}\ot \cdots \ot a_n\ns{0}.
\end{equation*}
Let $^{\Hc}C^n(A,M):=^{\Hc}\Hom(A^{\ot (n+1)}, M)$ denotes the set of all left $\mathcal{H}$-colinear morphisms. The following maps define a cocyclic module on $^{\Hc}C^*(A,M)$.

\begin{align}\label{comodule algebra}\nonumber
  &(\delta_i f)(a_0\ot \cdots \ot a_n)= f(a_0\ot \cdots \ot a_i a_{i+1}\ot \cdots \ot a_n), \quad 0\leq i< n,\\ \nonumber
  &(\delta_n f)(a_0\ot \cdots \ot a_n)=f(a_n\ns{0} a_0\ot a_1\cdots \ot a_{n-1})a_n\ns{-1},\\
  &(\sigma_i f)(a_0\ot \cdots \ot a_n)= f(a_0\ot \cdots \ot a_i \ot 1\ot \cdots \ot a_n), \quad 0\leq i< n,\\ \nonumber
  &(\tau_n f)(a_0\ot \cdots \ot a_n)=f(a_n\ns{0} \ot a_0\cdots \ot a_{n-1})a_n\ns{-1}.\nonumber
\end{align}

The cyclic cohomology of this cocyclic module is denoted by $^{\Hc}HC^n(A,M)$.
\begin{definition}
Let  $A$ be a left $\mathcal{H}$-comodule algebra. A  right-left module comodule $M$ over $\mathcal{H}$ is called an $^A\mathcal{H}$-SAYD module if for all   $a \in A$ and $\varphi \in \Hom^{\Hc}(A^{\ot (n+1)}, M)$  the following $^A\mathcal{H}$-AYD  and stability conditions are satisfied.
\begin{enumerate}
\item[i)]
\begin{align*}
&\left(\varphi(a\ns{0}\ot \widetilde{a})a\ns{-1}\right)\ns{-1}\ot \left(\varphi(a\ns{0}\ot \widetilde{a})a\ns{-1}\right)\ns{0}=\\
& S(a\ns{-1}\ps{3})\varphi(a\ns{0}\ot \widetilde{a})\ns{-1}a\ns{-1}\ps{1}\ot\varphi(a\ns{0}\ot \widetilde{a})\ns{0}a\ns{-1}\ps{2},
\end{align*}
where $\widetilde{a}=a_1\ot \cdots \ot a_n$.
\item [ii)] $\varphi(\widetilde{a}\ns{0})\widetilde{a}\ns{-1}= \varphi(\widetilde{a})$.
\end{enumerate}
\end{definition}
\begin{remark}\rm{
Obviously  for any $\mathcal{H}$-comodule algebra $A$,   any SAYD module over $\mathcal{H}$ is a $^A\mathcal{H}$-SAYD module.}
\end{remark}
For any character $\delta: \Hc\longrightarrow k$ where $k$ is the ground field of $\Hc$,  the twisted antipode $S_{\delta}: \Hc\longrightarrow \Hc$ is  defined by $S_{\delta}(h)= S(h\ps{2})\delta(h\ps{1})$.
The following lemma generalizes the notion of modular pair in involution \cite{ConMos:HopfCyc}.
\begin{lemma}
   Let $A$ be a left $\mathcal{H}$-comodule algebra,
    $\delta$ be a character and $\sigma$ be a group like element for $\Hc$. If $(\sigma, \delta)$ is a modular pair, \ie $\delta(\sigma)=1$,  and in $^A\Hc$-in involution, \ie
  \begin{equation}
\sigma^{-1} S_{\delta}^{2}(a\ns{-1}) \sigma  \ot a\ns{0}  =  a\ns{-1}  \ot a\ns{0}    ,
  \end{equation}
 for all $ h\in \mathcal{H}, a\in A,$ then $^\sigma\mathbb{C}_{\delta}$ is a $^A\mathcal{H}$-SAYD module.
\end{lemma}
\begin{proof}
For any $\varphi \in  \Hom^{\Hc}(A^{\ot (n+1)}, \mathbb{C})$   the following computation proves the  $^A\mathcal{H}$-AYD condition.
\begin{align*}
  &\left(\delta(a\ns{-1})\varphi(a\ns{0}\ot \widetilde{a})\right)\ns{-1}\ot \left(\delta(a\ns{-1})\varphi(a\ns{0}\ot \widetilde{b})\right)\ns{0}\\
  &= \delta(a\ns{-1})\varphi(a\ns{0}\ot \widetilde{a})\sigma\ot 1\\
  &=\delta(a\ns{-1}\ps{3})\varphi(a\ns{0}\ot \widetilde{a})\sigma S^{-1}(a\ns{-1}\ps{2})a\ns{-1}\ps{1}\ot 1\\
  &=\varphi(a\ns{0}\ot \widetilde{a}) \sigma S^{-1}_{\delta}(a\ns{-1}\ps{2} )\sigma^{-1}\sigma a\ns{-1}\ps{1}\ot 1\\
  &=\varphi(a\ns{0}\ot \widetilde{a}) S^{-1}_{\delta}(\sigma a\ns{-1}\ps{2} \sigma^{-1})\sigma a\ns{-1}\ps{1}\ot 1\\
  &=\varphi(a\ns{0}\ns{0}\ot \widetilde{a}) S_{\delta}^{-1}( S_{\delta}^{2}(a\ns{0}\ns{-1})\sigma a\ns{-1}\ot 1\\
  &=\varphi(a\ns{0}\ot \widetilde{a}) S_{\delta}^{-1}( S_{\delta}^{2}(a\ns{-1}\ps{2}))\sigma a\ns{-1}\ps{1}\ot 1\\
  &=\varphi(a\ns{0}\ot \widetilde{a})  S(a\ns{-1}\ps{3})\delta(a\ns{-1}\ps{2})\sigma a\ns{-1}\ps{1}\ot 1\\
  &=  S(a\ns{-1}\ps{3}) \varphi(a\ns{0}\ot \widetilde{a})\ns{-1}  a\ns{-1}\ps{1}\ot \varphi(a\ns{0}\ot \widetilde{a})\ns{0} a\ns{-1}\ps{2} .\\
  \end{align*}
  We use anti-algebra map property of $S^{-1}_{\delta}$ in the fourth equality,
   the coassociativity and $^A\Hc$-in involution condition for the element $a\ns{0}\in A$ in fifth equality. One has the $^A\mathcal{H}$-stability condition by the modular pair condition.
\end{proof}
\begin{proposition}
Let $(\d,\s)$ be a modular pair for  $\Hc$ and $ A$ be a left $\Hc$-comodule algebra. We define the following subspace of $A$,
\begin{align}\nonumber
B=\{ a\in A\mid \sigma^{-1} S_{\delta}^{2}(a\ns{-1}) \sigma  \ot a\ns{0}  =  a\ns{-1}  \ot a\ns{0}   \}.
\end{align}
Then $B$ is a $\Hc$-comodule subalgebra of $A$. Furthermore  $^\sigma\mathbb{C}_{\delta}$ is a  $^B\Hc$-SAYD module.
\end{proposition}
\begin{proof}
   Since $A$ is  an $\Hc$-comodule algebra and $S^{2}_{\delta}$  is an algebra map, the following computation shows that $B$ is a subalgebra of $A$. More precisely for any $a, b\in B$ we have
   \begin{align*}
   &\sigma^{-1} S_{\delta}^{2}((ab)\ns{-1}) \sigma  \ot (ab)\ns{0}= \sigma^{-1} S_{\delta}^{2}(a\ns{-1}) S_{\delta}^{2}(b\ns{-1})\sigma  \ot a\ns{0}b\ns{0}\\
   &=\sigma^{-1} S_{\delta}^{2}(a\ns{-1}) \sigma \sigma^{-1}S_{\delta}^{2}(b\ns{-1})\sigma \ot a\ns{0}b\ns{0}= a\ns{-1}b\ns{-1}\ot a\ns{0}b\ns{0}.
   \end{align*}
            To prove that $\Hc$ coacts on $B$, it is enough to show that for any $b\in B$ we have  $b\ns{-1}\ot b\ns{0}\in \Hc \ot B$,
            \begin{align*}
              &b\ns{-1}\ot \sigma^{-1}S_{\delta}^{2}(b\ns{0}\ns{-1}) \sigma \ot b\ns{0}\ns{0}\\
              &=\sigma^{-1} S_{\delta}^{2}(b\ns{-1}) \sigma\ot \sigma^{-1}S_{\delta}^{2}(b\ns{0}\ns{-1}) \sigma \ot b\ns{0}\ns{0}\\
              &=\sigma^{-1} S_{\delta}^{2}(b\ns{-1}\ps{1}) \sigma\ot \sigma^{-1}S_{\delta}^{2}(b\ns{-1}\ps{2}) \sigma \ot b\ns{0}\\
              &=\sigma^{-1} S_{\delta}^{2}(b\ns{-1})\ps{1} \sigma\ot \sigma^{-1}S_{\delta}^{2}(b\ns{-1})\ps{2} \sigma \ot b\ns{0}\\
              &=\left (\sigma^{-1} S_{\delta}^{2}(b\ns{-1}) \sigma\right )\ps{1}\ot \left(\sigma^{-1}S_{\delta}^{2}(b\ns{-1}) \sigma \right)\ps{2} \ot b\ns{0}\\
              & =b\ns{-1}\ps{1}\ot b\ns{-1}\ps{2}\ot b\ns{0}= b\ns{-1}\ot b\ns{0}\ns{-1}\ot b\ns{0}\ns{0}.
            \end{align*}
            We use   $b\in B$ in the first equality, the coassociativity of the coaction in the second equality  and  the coalgebra map property  of the map $S_{\delta}^{2}$ in the third equality.
\end{proof}
One notes that if $f\in ^{\Hc}C^n(A,M) $ then $f\mid_B \in ^{\Hc}C^n(B,M)$. This provides us a map
$^{\Hc}HC^n(B,M)\longrightarrow ^{\Hc}HC^n(A,M)$.

 \begin{definition}
 Let $\mathcal{H}$ be a Hopf algebra and $A$ be an algebra which is also an $\mathcal{H}$-comodule where the coaction of $\mathcal{H}$ on $A^{\ot (n+1)}$ is diagonal. A module-comodule $M$ over $\mathcal{H}$ is called an $_A\mathcal{H}$-Hopf cyclic coefficients, and abbreviated by $_A\mathcal{H}$-HCC,   if the cosimplicial and cyclic operators on $\Hom^{\Hc}(A^{\ot (n+1)}, M)$ are well-defined and make it a cocyclic module.
 \end{definition}

 \begin{proposition}
   Let  $M$ be a right-left module comodule over $\mathcal{H}$ and $A$ a right $\mathcal{H}$-comodule algebra. If $M$ is a $^A\mathcal{H}$-SAYD module, then $M$ is an $^A\mathcal{H}$-HCC.
 \end{proposition}
 \begin{proof}
   It is enough to show that the cyclic map is well-defined. The following computation proves that  $\tau f$ is a left $\mathcal{H}$-comodule map.
   \begin{align*}
    & \left((\tau_n f)(a_0\ot \cdots \ot a_n)\right)\ns{-1}\ot \left((\tau_n f)(a_0\ot \cdots \ot a_n)\right)\ns{0}\\
    &=\left(f(a_n\ns{0}\ot a_0\ot \cdots \ot a_{n-1})a_n\ns{-1}\right)\ns{-1}\ot \\
    &\left(f(a_n\ns{0}\ot a_0\ot \cdots \ot a_{n-1})a_n\ns{-1}\right)\ns{0}\\
    &=S(a_n\ns{-1}\ps{3})\left(f(a_n\ns{0}\ot a_0\ot \cdots \ot a_{n-1})\right)\ns{-1}a_n\ns{-1}\ps{2}\ot \\
    &\ot \left(f(a_n\ns{0}\ot a_0\ot \cdots \ot a_{n-1})\right)\ns{0}a_n\ns{-1}\ps{1}\\
    &=S(a_n\ns{-1}\ps{3})a_n\ns{0}\ns{-1}a_0\ns{-1}\cdots a_{n-1}\ns{-1}a_n\ns{-1}\ps{1}\ot \\
    &\ot f(a_n\ns{0}\ns{0}\ot a_0\ns{0}\ot \cdots \ot a_{n-1}\ns{0})a_n\ns{-1}\ps{2}\\
    &=S(a_n\ns{-1}\ps{3})a_n\ns{-1}\ps{4}a_0\ns{-1}\cdots a_{n-1}\ns{-1}a_n\ns{-1}\ps{1}\ot \\
    &\ot f(a_n\ns{0}\ot a_0\ns{0}\ot \cdots \ot a_{n-1}\ns{0})a_n\ns{-1}\ps{2}\\
    &=a_0\ns{-1}\cdots a_n\ns{-1}\ot f(a_n\ns{0}\ns{0}\ot a_0\ns{0}\ot \cdots \ot a_{n-1}\ns{0})a_n\ns{0}\ns{-1}\\
    &=a_0\ns{-1}\cdots a_n\ns{-1}\ot (\tau f)(a_0\ns{0}\ot \cdots \ot a_{n}\ns{0}).\\
   \end{align*}
   We use $^A\mathcal{H}$-AYD condition in the second equality, the $\mathcal{H}$-comodule map property of $f$ in the third equality and the coassociativity of the coaction for the element $a_n$ in the fourth and fifth equalities. To prove cyclicity, using $_A\mathcal{H}$-stability condition we have,
   \begin{align*}
     &\tau^{n+1}(a_0\ot \cdots \ot a_n)= f(a_0\ns{0}\ot \cdots \ot a_n\ns{0})a_0\ns{-1}\cdots a_n\ns{-1}\\
     &=f(\widetilde{a}\ns{0})\widetilde{a}\ns{-1}=f(a_0\ot \cdots \ot a_n).
   \end{align*}
   \end{proof}

     \begin{lemma}
  Let $A$ be a left $\Hc$-comodule algebra. If $\mathcal{H}$ coacts on $A$ commutatively, i.e.,
  \begin{equation}\label{cond7}
   {a}\ns{0}\ot {a}\ns{-1}g= {a}\ns{0} \ot g {a}\ns{-1}, \quad g \in \mathcal{H}, a\in A,
  \end{equation}
  then any comodule $M$ over $\mathcal{H}$ with a trivial action defines a $^A\mathcal{H}$-SAYD module.
  \end{lemma}
  \begin{proof}
  The following computation proves the $^A\mathcal{H}$-AYD condition.
    \begin{align*}
    &f(\widetilde{a}h)\ns{-1}\ot f(\widetilde{a}h)\ns{-1}=\ve(h) f(\widetilde{a})\ns{-1}\ot f(\widetilde{a})\ns{0}\\
    &=\ve(h)a\ns{-1}\ot f(\widetilde{a}\ns{0})=S(h\ps{2})h\ps{1}\widetilde{a}\ns{-1}\ot f(\widetilde{a}\ns{0})\\
    &=S(h\ps{3}) \widetilde{a}\ns{-1}h\ps{1}\ot \ve(h\ps{2})f(\widetilde{a}\ns{0})\\
    &=S(h\ps{3}) f(\widetilde{a})\ns{-1}h\ps{1}\ot f(\widetilde{a})\ns{0}h\ps{2}.
    \end{align*}
    We use \eqref{cond7} in the fourth equality. The counitality of the coaction implies the $^A\mathcal{H}$-stability condition.
  \end{proof}

  One notes that for any bicrossed product Hopf algebra $\Hc= \Fc\acl \Uc$, the Hopf algebra $\Fc$ is a right $\Hc$-comodule algebra by the coaction defined by
  \begin{align} \label{sym3}
  &f\longmapsto f\ps{1}\ot (f\ps{2}\acl 1_{\Uc}).
  \end{align}
  Here we introduce an example of a  commutative coaction.
\begin{example}\rm{
  Let $\Hc= \Fc \acl \Uc$ be a bicrossed Hopf algebra where $\Uc$ is not commutative. Suppose  $\Fc$ is commutative and $\Uc$ acts trivially on $\Fc$. The following computation shows that with respect to the coaction defined in \eqref{sym3}, $\Hc$ coacts commutatively on   $\Hc$-comodule algebra $\Fc$.
  \begin{align*}
  &{f}\ns{0}\ot {f}\ns{1}(g\lrbicross u)
 =f\ps{1} \ot (f\ps{2}\lrbicross 1)(g\lrbicross u)\\
  &=f\ps{1} \ot (f\ps{2}g\lrbicross u)
  =f\ps{1} \ot (gf\ps{2}\lrbicross u)\\
  &=f\ps{1}\ot  \ot ( gu\ps{1}\triangleright f\ps{2}\lrbicross u\ps{2})
  =f\ps{1} \ot (g\lrbicross u)(f\ps{2}\lrbicross 1)\\
  &{f}\ns{0}\ot (g\lrbicross u) {f}\ns{1}.\\
  \end{align*}
We use the commutativity of $\Fc$ in the third equality.
  }
\end{example}
\begin{lemma}
  Let $A$ be a left $\Hc$-comodule algebra. If $\mathcal{H}$ coacts on $A$ cocommutatively, \ie for any $a\in A$, $\widetilde{b}\in A^{\ot n}$ and $h\in \mathcal{H}$;
  \begin{equation}\label{cond13}
    \widetilde{b}\ns{-1}a\ns{-1}\ps{1}\ot a\ns{-1}\ps{2} \ot a\ns{0}\ot \widetilde{b}\ns{0} = a\ns{-1}\ps{2}\widetilde{b}\ns{-1} \ot  a\ns{-1}\ps{1}\ot a\ns{0}\ot \widetilde{b}\ns{0},
  \end{equation}
  then any module $M$ over $\mathcal{H}$ with the trivial coaction  defines a  $^A\mathcal{H}$-HCC module.
\end{lemma}
\begin{proof}
 It is enough to show that the cyclic map is  a left $\mathcal{H}$-comodule map.
\begin{align*}
  &(a_0\ot \cdots \ot a_n)\ns{-1}\ot (\tau_n f)\left(  (a_0\ot \cdots \ot a_n)\ns{0})\right)\\
  &=a_0\ns{-1}\cdots a_n\ns{-1}\ot (\tau_n f)(a_0\ns{0}\ot \cdots \ot a_n\ns{0})\\
  &=a_0\ns{-1}\cdots a_n\ns{-1}\ot f(a_n\ns{0}\ns{0}\ot a_0\ns{0} \cdots \ot a_{n-1}\ns{0})a_n\ns{0}\ns{-1}\\
  &=a_0\ns{-1}\cdots a_{n-1}\ns{-1}a_n\ns{-1}\ps{1}\ot f(a_n\ns{0}\ot a_0\ns{0} \cdots \ot a_{n-1}\ns{0})a_n\ns{-1}\ps{2}\\
  &=a_n\ns{-1}\ps{2}a_0\ns{-1}\cdots a_{n-1}\ns{-1}\ot f(a_n\ns{0}\ot a_0\ns{0} \cdots \ot a_{n-1}\ns{0})a_n\ns{-1}\ps{1}\\
  &=f(a_n\ns{0}\ot a_0\ot \cdots \ot a_{n-1})\ns{-1}\ot f(a_n\ns{0}\ot a_0\ot \cdots \ot a_{n-1})\ns{0}a_n\ns{-1}\\
  &=1 \ot f(a_n\ns{0}\ot a_0\ot \cdots \ot a_{n-1})a_n\ns{-1}\\
  &=\left(f(a_n\ns{0}\ot a_0\ot \cdots \ot a_{n-1})a_n\ns{-1}\right)\ns{-1}\ot \\
  &\ot \left(f(a_n\ns{0}\ot a_0\ot \cdots \ot a_{n-1})a_n\ns{-1}\right)\ns{0}\\
  &\left((\tau_n f)(a_0\ot \cdots \ot a_n)\right)\ns{-1}\ot \left((\tau_n f)(a_0\ot \cdots \ot a_n)\right)\ns{0}.
\end{align*}
 We use the relation \eqref{cond13} in the fourth equality,  the coassociativity of the coaction for the element $a_n$ and the left $\mathcal{H}$-comodule map property of the map $f$ in the fifth equality and the triviality of the coaction of $\mathcal{H}$ on $M$ in the sixth  and seventh equalities.
\end{proof}
Here we introduce an example of  a cocommutative coaction.
\begin{example}\rm{

Let $\Hc= \Fc \acl \Uc$ be a bicrossed product Hopf algebra where $\Fc$ is commutative and cocommutative. Then we consider the right $\Hc$-comodule algebra $\Fc$ with the coaction defined in \eqref{sym3}. The following computation shows that this coaction is cocommutative.
\begin{align*}
  &g\ns{0}\ot \widetilde{f}\ns{0}\ot \widetilde{f}\ns{1}g\ns{1}\ps{1}\ot g\ns{1}\ps{2} \\
  &=g\ns{0}\ot f_1\ns{0}\ot \cdots \ot f_n\ns{0}\ot f_1\ns{1}\cdots f_n\ns{1}g\ns{1}\ps{1}\ot g\ns{1}\ps{2}\\
  &=g\ps{1}\ot f_1\ps{1}\ot \cdots \ot f_n\ps{1}\ot f_1\ps{2}\cdots f_n\ps{2}g\ps{2}\ot g\ps{3}\\
 &=g\ps{1}\ot f_1\ps{1}\ot \cdots \ot f_n\ps{1}\ot g\ps{3}f_1\ps{2}\cdots f_n\ps{2}\ot g\ps{2}\\
 &=g\ns{0}\ot f_1\ns{0}\ot \cdots \ot f_n\ns{0}\ot g\ns{1}\ps{2} f_1\ns{1}\cdots f_n\ns{1}\ot g\ns{1}\ps{1}\\
 &g\ns{0}\ot \widetilde{f}\ns{0}\ot g\ns{1}\ps{2}\widetilde{f}\ns{1}\ot g\ns{1}\ps{1}. \\
\end{align*}
We use the commutativity and cocommutativity of $\Fc$ in the third equality.
  Therefore any module $M$ over $\mathcal{H}= \Fc \acl \Uc$ with the trivial coaction  defines a  $^{\Fc}{\Hc}$-HCC.
  One easily checks  that  since $\Fc$ is cocommutative any module $M$ with the trivial coaction in fact is a $^{\Fc}\Hc$-SAYD module. Therefore the $^{\Fc}\Hc$-HCC property of $M$ is not a result of  the cocommutativity of the coaction.
}
\end{example}

\begin{example}\rm{
In the special case of the previous example, consider $\Hc$ to be the co-opposite Hopf algebra of Schwarzian Hopf algebra $\Hc_{1s}^{cop}$  which is in fact the quotient of co-opposite Hopf algebra of Connes-Moscovici Hopf algebra $H^{cop}_1$ by the ideal $S$ generated by the Schwarzian  element $\d'_2= \d_2- \frac{1}{2}\d_1^{2}$. In fact $\Hc_{1s}^{cop}$ is generated by $X$, $Y$ and $Z=\d_1$ where
\begin{equation*}
  [Y,X]=X, \qquad [Y,Z]=Z, \qquad [X, Z]=\frac{1}{2}Z^2.
\end{equation*}
The coalgebra stucture and antipode are defined  similar to the one for $\Hc_1^{cop}$. The Hopf algebra $\Uc$ acts on $\Fc$ via
\begin{align*}
  X\triangleright Z= -\frac{1}{2}Z^2,  \qquad Y\triangleright Z= -Z,
\end{align*}
and $\Fc$ coacts on $\Uc$ via
\begin{align*}
  \blacktriangledown(X)= X\ot 1 + Y\ot Z, \qquad \blacktriangledown(Y)= Y\ot 1.
\end{align*}
Indeed we have $\Hc_{1s}^{cop}= \mathbb{C}[Z]\lrbicross \Uc$.
}
\end{example}
\begin{remark}
  \rm{Let $\mathcal{SAYD}_{\Hc}$ denotes the category of SAYD modules,  $^A\mathcal{H}\text{-}\mathcal{SAYD}$  denotes the category of  $^A\mathcal{H}$-SAYD modules and  $^A\mathcal{H}\text{-}\mathcal{HCC}$ denotes the category of $^A\mathcal{H}$-Hopf cyclic cohomology coefficients. We have seen commutative coactions as a source of $^A\mathcal{H}\text{-}\mathcal{SAYD}$  and  cocommutative coactions as a source of  $^A\mathcal{H}\text{-}\mathcal{HCC}$. Based on our arguments  in this subsection, one has  the following proper inclusions of categories,
  \begin{equation*}
  {\mathcal{SAYD}_{\Hc} \quad \subsetneqq   \quad  ^A\mathcal{H}\text{-}\mathcal{SAYD} \quad    \subsetneqq    \quad ^A\mathcal{H}\text{-}\mathcal{HCC}.}
\end{equation*}

  }
\end{remark}

    \subsection{The $\mathcal{H}^C$-SAYD and HCC modules for comodule coalgebras}
In this subsection  we generalize the  Hopf cyclic  cohomology of  comodule coalgebras with coefficients.
 Let $C$ be a right $\mathcal{H}$-comodule coalgebra and $M$ a right-left SAYD module on $\mathcal{H}$. We set
   \begin{equation}
     ^{\mathcal{H}}C^{n}(C, M)= C^{\ot(n+1)}\square_\mathcal{H} M.
   \end{equation}
   The following  maps define a cocyclic module for $^{\mathcal{H}}C^{n}(C, M)$.
   \begin{align}\label{comodule-coalgebra}\nonumber
    &\d_i( c_0\ot \cdots \ot c_n\ot m)=  c_0\ot\cdots \ot \Delta(c_i)\ot \cdots\ot c_n\ot m,\\ \nonumber
    &\d_n( c_0\ot \cdots \ot c_n\ot m)=c_0\ps{2}\ot c_1\ot\cdots\ot c_n\ot c_0\ps{1}\ns{0}\ot m c_0\ps{1}\ns{1},\\
    &\s_i( c_0\ot \cdots \ot c_n\ot m)= c_0\ot\cdots \ot \varepsilon(c_{i+1})\ot \cdots\ot c_n\ot m,\\ \nonumber
    &\tau_n( c_0\ot \cdots \ot c_n\ot m)= c_1\ot \cdots \ot c_n\ot c_0\ns{0}\ot m c_0\ns{1}.
   \end{align}
   Here $C^{\ot(n+1)}$ is a right $\mathcal{H}$-comodule by diagonal coaction.

\begin{definition}
Let $\mathcal{H}$ be a Hopf algebra and $C$ be a right $\mathcal{H}$-comodule coalgebra. A  right-left module-comodule $M$ over $\mathcal{H}$ is called an $\mathcal{H}^C$-SAYD module if the following $\mathcal{H}^C$-AYD  and stability conditions are satisfied.
\begin{enumerate}
\item[i)] $ c\ns{0}\ot (m c\ns{1})\ns{-1}\ot (m c\ns{1})\ns{0}= \\
 c\ns{0}\ot S(c\ns{1}\ps{3})m\ns{-1}c\ns{1}\ps{1}\ot m\ns{0}c\ns{1}\ps{2},$
\item [ii)] $\widetilde{d}\ns{0}\square_{\Hc} m \widetilde{d}\ns{1}= \widetilde{d}\square_{\Hc} m$,
\end{enumerate}
where $\widetilde{d}=c_0\ot \cdots \ot c_n$.
\end{definition}
   One notices that  since  $\widetilde{d}\ot m\in C^{\ot (n+1)}\square_\mathcal{H} M$, then stability condition is equivalent to
   $$\widetilde{d}\square_{\Hc} m\ns{0}m\ns{-1}= \widetilde{d} \square_{\Hc} m.$$

  It is clear that any SAYD module over $\mathcal{H}$  implies a $\mathcal{H}^C$-SAYD module.

\begin{lemma}\label{first3}
 Let $C$ be a right $\mathcal{H}$-comodule coalgebra.
   If $(\sigma, \delta)$ be a modular pair and $\Hc^C$-in involution, \ie
  \begin{equation}\label{involution-cc}
 c\ns{0}\ot   S^2_{\delta}(c\ns{1})=c\ns{0}\ot \sigma c\ns{1}\sigma^{-1}, \qquad  c\in C,
  \end{equation}
    then $^\sigma\mathbb{C}_{\delta}$ is a $\mathcal{H}^C$-SAYD module.
  \end{lemma}
  \begin{proof} The following computation proves the $\Hc^C$-AYD condition,
  \begin{align*}
   &  c\ns{0}\ot \sigma \delta( c\ns{1})\ot 1\\
   &=  c\ns{0}\ot \sigma S^{-1}(c\ns{1}\ps{2})\delta(c\ns{1}\ps{3})c\ns{1}\ps{1}\ot 1\\
   &= c\ns{0}\ot \sigma S_{\delta}^{-1}(c\ns{1}\ps{2})\sigma^{-1} \sigma c\ns{1}\ps{1}\ot 1\\
   &= c\ns{0}\ot S_{\delta}^{-1}(\sigma c\ns{1}\ps{2}\sigma^{-1}) \sigma c\ns{1}\ps{1}\ot 1\\
   &=  c\ns{0}\ns{0}\ot S_{\delta}^{-1}( S_{\delta}^2(c\ns{1}))\sigma c\ns{0}\ns{1}\ot 1\\
   &=  c\ns{0}\ot S_{\delta}(c\ns{1}\ps{2})\sigma c\ns{1}\ps{1}\ot 1\\
   &= c\ns{0}\ot S(c\ns{1}\ps{3})\sigma c\ns{1}\ps{1}\ot \delta(c\ns{1}\ps{2}).
    \end{align*}
    We use the coassociativity of the coaction and  $\Hc^C$-in involution property in the fourth equality. The stability condition is obvious by the modular pair property.
  \end{proof}


\begin{lemma}
  If $\mathcal{H}$ coacts on $C$ commutatively, \ie
  \begin{equation}\label{cond8}
    c\ns{0}\ot hc\ns{1}= c\ns{0}\ot c\ns{1}h
  \end{equation}
  then any $\mathcal{H}$-comodule $M$  with trivial action becomes a $\mathcal{H}^C$-SAYD module.
\end{lemma}
\begin{proof}
The $\mathcal{H}^C$-stability condition is obvious.The following computation proves the $\mathcal{H}^C$-AYD condition.
  \begin{align*}
    &c\ns{0}\ve( c\ns{1})\ot m\ns{-1}\ot m\ns{0}\\
    &= c\ns{0}\ot S(c\ns{1}\ps{2})c\ns{1}\ps{1}m\ns{-1}\ot m\ns{0}\\
    &=c\ns{0}\ns{0}\ot S(c\ns{1})c\ns{0}\ns{1}m\ns{-1}\ot m\ns{0}\\
    &=c\ns{0}\ns{0}\ot S(c\ns{1})m\ns{-1}c\ns{0}\ns{1}\ot m\ns{0}\\
    &=c\ns{0}\ot S(c\ns{1}\ps{3})m\ns{-1}c\ns{1}\ps{1}\ot m\ns{0}\ve(c\ns{1}\ps{2})\\
    &=c\ns{0}\ot S(c\ns{1}\ps{3})m\ns{-1}c\ns{1}\ps{1}\ot m\ns{0}c\ns{1}\ps{2}.
  \end{align*}
  We use the \eqref{cond8} in the third equality and coassociativity of the coaction in the fourth equality.
\end{proof}
Let $\Hc=\Fc \acl \Uc$ be any bicrossed product Hopf algebra  where $u\longmapsto u\ns{0}\ot u\ns{1}$ denotes the coaction of $\Fc$ on $\Uc$.
  The Hopf algebra  $\Uc$ is a $\Hc$-comodule coalgebra by the following coaction.
  \begin{align}\label{com7}
    u\longmapsto u\ns{0}\ot u\ns{1}\ot 1_{\Uc}.
  \end{align}

\begin{example}\rm{
 Let $\Hc= \Fc\acl \Uc$ be a bicrossed product Hopf algebra. If $\Fc$ is commutative  and $\Uc$ acts trivially on $\Fc$, then $\Hc$ coacts commutatively on $\Uc$.}
\end{example}

 \begin{definition}
 Let $\mathcal{H}$ be a Hopf algebra and $C$ be a coalgebra, which is also a $\mathcal{H}$-comodule where the coaction of $\mathcal{H}$ on $C^{\ot (n+1)}$ is diagonal. A module-comodule $M$ over $\mathcal{H}$ is called a $\mathcal{H}^C$-Hopf cyclic coefficients and abbreviated by $\mathcal{H}^C$-HCC,   if the cosimplicial and cyclic operators on $C^{\ot(n+1)}\square_\mathcal{H} M$ are well-defined and make it a cocyclic module.
 \end{definition}

 \begin{proposition}
   Let  $M$ be a right-left module-comodule over $\mathcal{H}$ and  $C$ be a right $\mathcal{H}$-comodule coalgebra. If $M$ is a $\mathcal{H}^C$-SAYD module then $M$ is an $\mathcal{H}^C$-HCC .
 \end{proposition}
\begin{proof}
       It is enough to show that the cyclic map $\tau$ is  well-defined. The following computation shows that $\tau_n(\widetilde{c}\ot m)\in ~ ^{\mathcal{H}}C^{n}(C, M)$.

\begin{align*}
  &c_1\ot \cdots \ot c_n\ot c_0\ns{0}\ot (m c_0\ns{1})\ns{-1}\ot (m c_0\ns{1})\ns{0}\\
  &=c_1\ot \cdots \ot c_n\ot c_0\ns{0}\ot  S(c_0\ns{1}\ps{3})m\ns{-1}c_0\ns{1}\ps{1}\ot m\ns{0}c_0\ns{1}\ps{2}\\
  &=c_1\ns{0}\ot \cdots \ot c_n\ns{0}\ot c_0\ns{0}\ns{0}\ot  S(c_0\ns{0}\ns{1}\ps{3})c_0\ns{1}\cdots c_n\ns{1}c_0\ns{0}\ns{1}\ps{1}\ot \\
  &~~~~~~~~~~~~~~~~~~~~~~~~~~~~~~~~~~~~~~~~~~~~~~~~~~~~~~~~~~~~~~~~~~~~~~~~~~\ot mc_0\ns{0}\ns{1}\ps{2}\\
  &=c_1\ns{0}\ot \cdots \ot c_n\ns{0}\ot c_0\ns{0}\ot  S(c_0\ns{1}\ps{3})c_0\ns{1}\ps{4}c_1\ns{1}\cdots c_n\ns{1}c_0\ns{1}\ps{1}\ot\\
  &~~~~~~~~~~~~~~~~~~~~~~~~~~~~~~~~~~~~~~~~~~~~~~~~~~~~~~~~~~~~~~~~~~~~~~~~~~\ot mc_0\ns{1}\ps{2}\\
  &=c_1\ns{0}\ot \cdots \ot c_n\ns{0}\ot c_0\ns{0}\ot c_1\ns{1}\cdots c_n\ns{1}c_0\ns{1}\ps{1}\ot mc_0\ns{1}\ps{2}\\
  &=c_1\ns{0}\ot \cdots \ot c_n\ns{0}\ot c_0\ns{0}\ns{0}\ot c_1\ns{1}\cdots c_n\ns{1}c_0\ns{0}\ns{1}\ot mc_0\ns{1}.\\
\end{align*}
We use the $\mathcal{H}^C$-AYD condition in the first equality, the relation $c_0\ot \cdots \ot c_n\ot m\in C^{\ot(n+1)}\square_\mathcal{H} M $ in the second equality, the coassociativity of the coaction for the element $a_0$ in the third  and last equalities. To prove the cyclicity, using the $\mathcal{H}^C$-stability condition we have,
\begin{align*}
  &\tau_n^{n+1}(c_0\ot \cdots \ot c_n\ot m)= c_0\ns{0}\ot \cdots \ot c_n\ns{0}\ot m c_0\ns{1}\cdots c_n\ns{1}\\
  &=c_0\ot \cdots \ot c_n\ot m\ns{0}m\ns{-1} =c_0\ot \cdots \ot c_n\ot m.
\end{align*}
\end{proof}
\begin{lemma}
Let  $C$ be a $\Hc$-comodule coalgebra. If the coaction of $\mathcal{H}$ on $C$ is cocommutative, \ie
\begin{equation}\label{cond10}
   \widetilde{c}\ns{0}\ot d\ns{0}\ot  \widetilde{c}\ns{1} d\ns{1}\ps{1}\ot d\ns{1}\ps{2}= \widetilde{c}\ns{0}\ot d\ns{0}\ot d\ns{1}\ps{2} \widetilde{c}\ns{1}\ot d\ns{1}\ps{1},
\end{equation}
for  all $c\in C$ and $h\in \mathcal{H}$, then any module $M$ over $\mathcal{H}$ with the trivial coaction  defines a $\mathcal{H}^C$-HCC.
\end{lemma}
\begin{proof}
  It is enough to show that the cyclic map is well-defined. The following computation proves $\tau(\widetilde{c}\ot m) \in C^{\ot(n+1)}\square_H M$.
  \begin{align*}
    &c_1\ns{0}\ot \cdots \ot c_n\ns{0}\ot c_0\ns{0}\ns{0}\ot c_1\ns{1}\cdots c_n\ns{1}c_0\ns{0}\ns{1}\ot mc_0\ns{1}\\
    &=c_1\ns{0}\ot \cdots \ot c_n\ns{0}\ot c_0\ns{0}\ot c_1\ns{1}\cdots c_n\ns{1}c_0\ns{1}\ps{1}\ot mc_0\ns{1}\ps{2}\\
    &=c_1\ns{0}\ot \cdots \ot c_n\ns{0}\ot c_0\ns{0}\ot c_0\ns{1}\ps{2} c_1\ns{1}\cdots c_n\ns{1}\ot mc_0\ns{1}\ps{1}\\
    &=c_1\ot \cdots \ot c_n\ot c_0\ns{0}\ot 1 \ot m c_0\ns{1}.
  \end{align*}
  We use \eqref{cond10} in the second equality and $c_0\ot \cdots \ot c_n \ot m \in C^{\ot (n+1)}\square_{\mathcal{H}} M$ in the last equality.
\end{proof}

\begin{example}\rm{

Let $\Hc= \Fc \acl \Uc$ be a bicrossed product Hopf algebra where $\Fc$ is commutative and cocommutative.  The following computation proves that the coaction of $\Hc$ on $\Uc$  defined in \eqref{com7} is  cocommutative.
\begin{align*}
& \widetilde{u}\ns{0}\ot d\ns{0}\ot  \widetilde{u}\ns{1} d\ns{1}\ps{1}\ot d\ns{1}\ps{2}\\
 &=u_1\ns{0}\ot \cdots \ot u_n\ns{0}\ot d\ns{0}\ot u_1\ns{1}\cdots u_n\ns{1} d\ns{1}\ps{1}\ot d\ns{1}\ps{2}\\
 &=u_1\ns{0}\ot \cdots \ot u_n\ns{0}\ot d\ns{0}\ot (u_1\ns{1}\lrbicross 1) \cdots (u_n\ns{1}\lrbicross 1)( d\ns{1}\ps{1}\lrbicross 1)   \ot \\
 &~~~~~~~~~~~~~~~~~~~~~~~~~~~~~~~~~~~~~~~~~~~~~~~~~~~~~~~~~\ot (d\ns{1}\ps{2}\lrbicross 1)\\
 &=u_1\ns{0}\ot \cdots \ot u_n\ns{0}\ot d\ns{0}\ot (u_1\ns{1}\cdots u_n\ns{1}d\ns{1}\ps{1}\lrbicross 1)  \ot (d\ns{1}\ps{2}\lrbicross 1)  \\
 &=u_1\ns{0}\ot \cdots \ot u_n\ns{0}\ot d\ns{0}\ot (d\ns{1}\ps{2}u_1\ns{1}\cdots u_n\ns{1}\lrbicross 1)  \ot (d\ns{1}\ps{1}\lrbicross 1)\\
 &=u_1\ns{0}\ot \cdots \ot u_n\ns{0}\ot d\ns{0}\ot ( d\ns{1}\ps{2}\lrbicross 1) (u_1\ns{1}\lrbicross 1) \cdots (u_n\ns{1}\lrbicross 1)  \ot \\
 &~~~~~~~~~~~~~~~~~~~~~~~~~~~~~~~~~~~~~~~~~~~~~~~~~~~~~~~~~\ot (d\ns{1}\ps{1}\lrbicross 1)\\
 &=\widetilde{u}\ns{0}\ot d\ns{0}\ot d\ns{1}\ps{2} \widetilde{u}\ns{1}\ot d\ns{1}\ps{1}.
\end{align*}
 In the fourth equality we use  the commutativity and cocommutativity of $\Fc$. As an example, co-opposite Hopf algebra of Schwarzian Hopf algebra $\Hc_{1s}^{cop}$ coacts cocommutatively on $\Uc$.
}
\end{example}

Let   $^C\mathcal{H}\text{-}\mathcal{SAYD}$  denotes the category of  $^C\mathcal{H}$-SAYD modules and  $^C\mathcal{H}\text{-}\mathcal{HCC}$ denotes the category of $^C\mathcal{H}$-Hopf cyclic cohomology coefficients. It is shown that  commutative coactions are a source of $^C\mathcal{H}\text{-}\mathcal{SAYD}$  and  cocommutative coactions as a source of  $^C\mathcal{H}\text{-}\mathcal{HCC}$. Based on our arguments  in this subsection, one has  the following proper inclusions of categories,

  \begin{equation*}
  {\mathcal{SAYD}_{\Hc} \quad \subsetneqq   \quad  ^C\mathcal{H}\text{-}\mathcal{SAYD} \quad    \subsetneqq    \quad ^C\mathcal{H}\text{-}\mathcal{HCC}.}
\end{equation*}

\section{Pairing between module algebras and comodule algebras}

 In this section we generalize the pairing between Hopf cyclic cohomology of module algebras and comodule algebras \cite{atabey2}, \cite{bm2}. Let $\mathcal{H}$ be a Hopf algebra, $A$ a left $\mathcal{H}$-module algebra, $B$ a left $\mathcal{H}$-comodule algebra and  $\mathcal{M}$  a right-left $_A\mathcal{H}$- and $^B\Hc$-SAYD module. We consider the crossed product algebra $A\rtimes B$ with the following multiplication,
\begin{align}
  (a\rtimes b)(a'\rtimes b')= a b\ns{-1} a'\rtimes b\ns{0} b'.
\end{align}
This is an unital algebra where its unit element is $1\rtimes 1$. Let $C^{\Hc}(A,M):=\text{Hom}^{\Hc}(A^{\ot (n+1)}, M)$ and $C^n_{\Hc}(A, M)= \text{Hom}_{\Hc}(M\ot A^{\ot(n+1)}, \mathbb{C})$ be the cocyclic modules defined in   \eqref{comodule algebra} and \cite{HaKhRaSo2} respectively.
We consider the following diagonal complex,
\begin{equation}
  C^{n,n}_{a-a}:= \text{Hom}_{\Hc}(M\ot A^{\ot(n+1)}, \mathbb{C}) \ot \text{Hom}^{\Hc}(B^{\ot (n+1)}, M)
\end{equation}
with the cocyclic structure $(\delta\times d , \sigma\times s ,\tau\times t)$.
We define the following map,
\begin{align*}
 &\Psi: C^{n,n}_{a,a}\longrightarrow \Hom( (A\rtimes B)^{\ot (n+1)}, \mathbb{C}) )\\
 &\Psi(\varphi \ot \psi)(a_0\rtimes b_0 \ot \cdots \ot a_n\rtimes b_n)=\\
 & \varphi( \psi(b_0\ns{0}\ot \cdots b_n\ns{0})\ot S^{-1}(b_0\ns{1}\cdots b_n\ns{-1})a_0\ot \cdots S^{-1}(b_n\ns{-n-1}a_n)).
\end{align*}
We have the following lemma.
\begin{lemma}
  The map $\Psi$ is a cocyclic  map between cocyclic modules  $C^{*,*}$ and $C^*(A\rtimes B)$.
\end{lemma}
This shows that  $\Psi$ induces a map on the level of cyclic cohomology. One use the $^B\Hc$-SAYD module property to
use the similar argument in \cite{rangipour} for proving the following proposition.
\begin{proposition}
Let $\mathcal{H}$ be a Hopf algebra, $A$ a left $\mathcal{H}$-module algebra, $B$ a left $\mathcal{H}$-comodule
algebra and  $M$  a right-left $_A\mathcal{H}$ and $^B\Hc$-SAYD module. The following map defines a cup product on Hopf cyclic cohomology.
  \begin{align}\nonumber
  &\sqcup: HC_{\mathcal{H}}^p(A, M)\ot HC_{\mathcal{H}}^q(B,M)\longrightarrow HC^{p+q}(A\rtimes B),\\
   &~~~~~~~~~~~~~~~~~~~~~~~~~~ \sqcup:= \Psi \circ AW,
  \end{align}
  where AW is Alexander-Whitney map.
\end{proposition}


  \end{document}